\documentclass[a4paper,leqno,11pt]{amsart}
\usepackage{amsfonts,amssymb,verbatim,amsmath,amsthm,latexsym,textcomp,amscd}
\usepackage{latexsym,amsfonts,amssymb,epsfig,verbatim}
\usepackage{amsmath,amsthm,amssymb,latexsym,graphics,textcomp}
\usepackage{graphicx}
\usepackage[usenames,dvipsnames]{color}
\usepackage{url}
\usepackage{enumerate}
\usepackage[mathscr]{euscript}

\usepackage{tikz}
\usepackage{dsfont}
\usetikzlibrary{matrix}
\usepackage{sseq}
\usepackage[colorlinks,citecolor=red, linkcolor=blue]{hyperref}

\input xy
\xyoption{all}

\theoremstyle{plain}
\newtheorem*{thma}{Theorem A}
\newtheorem*{thmb}{Theorem B}
\newtheorem*{thmc}{Theorem C}

\theoremstyle{definition}
\newtheorem{theorem}{Theorem}[section]
\newtheorem{prop}[theorem]{Proposition}
\newtheorem{lemma}[theorem]{Lemma}

\newtheorem{defn}[theorem]{Definition}

\newtheorem{rmk}[theorem]{Remark}

\newtheorem{exam}[theorem]{Example}

\newtheorem{subsec}[theorem]{}
\newtheorem{thm}[theorem]{Theorem}

\theoremstyle{remark}
\swapnumbers
   \newtheorem{ack}[theorem]{Acknowledgements}
   
 \theoremstyle{definition}

\newcommand{\C}{{\mathbb C}}

\newcommand{\Z}{{\mathbb{Z}}}
\newcommand{\R}{{\mathbb R}}

\newcommand{\Hom}{\mathit{Hom}}

\newcommand{\upi}{\underline{\pi}}

\newcommand\DD{{\mathcal D}}

\newcommand\FF{{\mathcal F}}
\newcommand\GG{{\mathcal G}}

\newcommand\LL{{\mathcal L}}
\newcommand\MM{{\mathcal M}}

\newcommand\OO{{\mathcal O}}
\newcommand\PP{{\mathcal P}}

\newcommand\PMF{{\PP\kern-2pt\MM\FF}}

\newcommand\PML{{\PP\kern-2pt\MM\LL}}

\newcommand\Mod{\operatorname{Mod}}

\newcommand\tr{\operatorname{tr}}

\newcommand{\ssm}{\setminus}

\newcommand{\T}{{\mathbb T}}

\newcommand{\fsubd}{\mathrel{{\scriptstyle\searrow}\kern-1ex^d\kern0.5ex}}
\newcommand{\bsubd}{\mathrel{{\scriptstyle\swarrow}\kern-1.6ex^d\kern0.8ex}}
\newcommand{\fsubeq}{\mathrel{\raise-.7ex\hbox{$\overset{\searrow}{=}$}}}
\newcommand{\bsubeq}{\mathrel{\raise-.7ex\hbox{$\overset{\swarrow}{=}$}}}

\newcommand{\tsh}[1]{\left\{\kern-.9ex\left\{#1\right\}\kern-.9ex\right\}}

\newenvironment{myeq}[1][]
{\stepcounter{theorem}\begin{equation}\tag{\thetheorem}{#1}}
	{\end{equation}}

\newenvironment{mysubsection}[2][]
{\begin{subsec}\begin{upshape}\begin{bfseries}{#2.}
			\end{bfseries}{#1}}
		{\end{upshape}\end{subsec}}

\usepackage{tikz-cd}

\newcommand{\bZp}{\langle \Z/p \rangle}

\newcommand{\res}{\mathit{res}}

\newcommand{\uA}{\underline{A}}

\newcommand{\uM}{\underline{M}}

\newcommand{\up}[1]{\;\textup{#1}\;}

\newtheorem{notation}[theorem]{Notation}

\tikzcdset{scale cd/.style={every label/.append style={scale=#1},
		cells={nodes={scale=#1}}}}

\newcommand{\cp}{C_{p}}

\newcommand{\alpabs}[1]{\vert\alpha^{#1}\vert}

\usepackage{ushort}
\newcommand{\uH}{\ushort{H}}

\newcommand{\uZ}{\underline{\mathbb{Z}}}

\usepackage[normalem]{ulem}

\numberwithin{equation}{section}

\usepackage{todonotes}
\tikzset{notestyleraw/.append style={align=justify}}
\newcommand{\smas}{\wedge}
\newcommand{\bigwedg}{\bigvee}

\usepackage[normalem]{ulem}
\newcommand{\cofseq}{cofibre sequence}

\usepackage{epsfig}

\makeatletter
\@namedef{subjclassname@2020}{%
  \textup{2020} Mathematics Subject Classification}
\makeatother

\title[Equivariant homology decompositions]{Equivariant homology decompositions  for cyclic group actions on definite 4-manifolds}

\author{Samik Basu, Pinka Dey, Aparajita Karmakar}

\email{samik.basu2@gmail.com; samikbasu@isical.ac.in}
\address{Stat-Math Unit,
	Indian Statistical Institute,
	B. T. Road, Kolkata-700108, India.}

\email{pinkadey11@gmail.com;}
\address{Stat-Math Unit,
	Indian Statistical Institute,
	B. T. Road, Kolkata-700108, India.}

\email{aparajitakarmakar@gmail.com;}
\address{Stat-Math Unit,
	Indian Statistical Institute,
	B. T. Road, Kolkata-700108, India.}


\begin{document}
	\maketitle
	\begin{abstract}
In this paper, we study the equivariant homotopy type of a connected sum of linear actions on complex projective planes defined by Hambleton and Tanase. These actions are constructed for cyclic groups of odd order. We construct cellular filtrations on the connected sum using spheres inside unitary representations. A judicious choice of filtration implies a splitting on equivariant homology for general cyclic groups under a divisibility hypothesis, and in all cases for those of prime power order. 		
	\end{abstract}	
	\section{Introduction}
	
	Simply connected $4$-manifolds form an important category of spaces from the point of view of both topologists and geometers. Their homotopy type is determined by the intersection form. The ones with positive definite intersection form are homotopy equivalent to a connected sum of copies of $\C P^2$. 
This paper studies the equivariant homotopy type of certain cyclic group actions on these $4$-manifolds defined in \cite{HT04}, and splitting results for the equivariant homology with constant coefficients.  

Recall that a simply connected $4$-manifold $M$ possesses a CW-complex structure whose $2$-skeleton is a wedge of spheres, and outside the $2$-skeleton, there is a single $4$-cell. It follows  that the homology is torsion-free, and non-zero in only three degrees $0$, $2$ and $4$, with $H_0(M)=\Z$ and $H_4(M)=\Z$. If $k$ is the second Betti number of $M$, in the stable homotopy category, we obtain the decomposition $H\Z \wedge M_+ \simeq H\Z \bigvee_{i=1}^k \Sigma^2 H\Z \bigvee \Sigma^4 H\Z$.   

Homology decompositions akin to the above in equivariant stable homotopy category have been studied for complex projective spaces in \cite{Lew88} over the group $C_p$. The splitting therein is obtained with Burnside ring coefficients, which is denoted by $\uA$. However, Lewis \cite[Remark 2.2]{Lew92} shows that these decompositions are not likely if the group contains $C_{p^2}$ or $C_p\times C_p$. On the other hand, for cyclic groups of square free order these splittings do exist (\cite{BG19}, \cite{BG20}). There are generalizations of these results for the group $C_p$ for even cell complexes (\cite{Fer99}, \cite{FL04}), and with constant $\Z/p$-coefficients \cite{BG21a}. Such results have also been extensively studied for the group $C_2$ with constant $\Z/2$-coefficients (\cite{May20}, \cite{Kro10}, \cite{Haz21}). In this paper, we prove decomposition results for cyclic group actions on a connected sum of copies of $\C P^2$ with constant $\Z$-coefficients.     

The homology decompositions for $X$ are usually proved by building up a cellular filtration of $X$, and then showing that after smashing with the spectrum $H\uZ$, the connecting maps are all trivial. For this purpose, the cells are taken of the form $D(V)$, a disk in a unitary $G$-representation $V$, so that the filtration quotients are wedges of $S^V$, the one-point compactification of $V$. There is a different concept of $G$-CW-complexes with cells of the type $G/H\times \DD^n$ for  subgroups $H\leq G$, but they are not useful from the point of view of homology decompositions. 

In this paper, the $G$-manifolds $X(\T)$ are defined using {\it admissible weighted trees} $\T$ \cite{HT04}, which are directed rooted trees with $G$-action, with each vertex carrying a weight comprising $3$ integers $a, b, m$ of $\gcd$ $1$ (see figure below). The underlying manifolds $\C P^2(a,b;m)$ are copies of $\C P^2$, which have an action of the group $C_m$, by identifying $\C P^2$ as the space of complex lines in a three-dimensional complex representation of $C_m$. The numbers $a,b$ are used to describe the irreducible representations therein.  We fix $\lambda$ as a complex $1$-dimensional representation where $C_m$ acts via $m^{th}$-roots of unity, and in these terms $\C P^2(a,b;m)=P(\lambda^a+\lambda^b+1_\C)$. The admissible part of the definition of the tree allows us to construct the equivariant connected sum in the figure. 
 \addtocounter{equation}{1}
{\scriptsize
		 \begin{figure}[htp]
		 	\begin{tikzpicture}
		 		[
		 		scale=.35,
		 		dot/.style = {circle, gray, minimum size=#1,
		 			inner sep=0pt, outer sep=0pt}, 
		 		dot/.default = 3pt  
		 		] 
		 		
		 		\node [dot, label=left:\up{$v_0$}] (a0) at (3,6)  {};  
		 		\node [dot, label=right:\up{$w(v_0)=(a_0,b_0;m)$}] (a0) at (3,6)  {};
		 		\draw[black] (3,6) circle[radius=3pt];
		 		\node [dot, label=left:\up{$v_1$}](a1) at (1,4)  {};
		 		\node [dot, label=right:\up{$w(v_1)=(a_1,b_1;m)$}](a1) at (1,4)  {};
		 		\draw[black] (1,4) circle[radius=3pt];
		 		\node [dot, label=left:\up{$v_2$\;\;\;}] (a2) at (3,2) {}; 
		 		\node [dot, label=right:\up{$w(v_2)=(a_2,b_2;m)$\;\;\;}] (a2) at (3,2) {};  
		 		\draw[black] (3,2) circle[radius=3pt];
		 		\node [dot, label=left:\up{$v_3$}] (a3) at (1,0) {}; 
		 		\node [dot, label=right:\up{$w(v_3)=(a_3,b_3;m)$}] (a3) at (1,0) {}; 
		 		\draw[black] (1,0) circle[radius=3pt];

		 		\draw[->](a0) --  (a1); 
		 		\draw[->] (a1) -- (a2);  
		 		\draw[->] (a2) -- (a3);
		 		\node[dot,label=left:\up{\normalsize $\T$}](a4) at (3.5,-2){ };  
		 	\end{tikzpicture}
		 	\hspace*{1.5cm}
		 	\begin{tikzpicture}  
		 		[
		 		scale=.35,
		 		dot/.style = {circle, gray, minimum size=#1,
		 			inner sep=0pt, outer sep=0pt}, 
		 		dot/.default = 3pt  
		 		]  
		 		
		 		\node [dot, label=left:\up{\normalsize $X(\T)$}] (a0) at (4,10)  {};

		 		\node [dot](a1) at (4.5,10)  {};
		 		
		 		\node [dot] (a2) at (6.5,10) {};  
		 		
		 		\node [dot, label=right:\up{ $\C P^2(a_0,b_0;m)\# \C P^2(a_1,b_1;m)$}] (a3) at (7,10.5) {}; 
		 		
		 		\node [dot, label=right:\up{ $\#\C P^2(a_2,b_2;m)\#\C P^2(a_3,b_3;m)$}] (a3) at (7,9.3) {}; 
		 		\node[dot](a4) at (7,6) { };
		 		
		 		\draw[double distance=2pt] (a1) -- (a2);  
		 	\end{tikzpicture}
		 	\label{fig-tree}
		 	\caption{ } 
		 \end{figure}
		 }

For a cyclic group $C_m$ of odd order, we prove two decomposition results (see Theorem \ref{thm:gen} and Theorem \ref{thm:one-zero})
\begin{thma}
a) If $ \T $  is an admissible weighted tree such that all fixed vertices $ v$ with weight $w(v)= (a_v,b_v;m_v) $ satisfy  $\gcd(a_v-b_v,m_v)=1$, then  
\[
		H\uZ\smas X(\T)_+\simeq
			H\uZ\bigwedg H\uZ \smas S^{\lambda ^{a_0}+ \lambda^{b_0}} \underset{\T_0}{\bigvee} H\uZ\wedge S^\lambda \underset{[\mu]\in\T_d/\GG, d\neq m}{\bigvee} { H\uZ \smas \GG/C_d}_+\smas S^{\lambda^{a_\mu-b_\mu}} 
		\]
where $w(v_0)= (a_0,b_0;m) $ is the weight of the root vertex. \\
b) Let $ \T $ be an admissible weighted tree such that for the root vertex $ v_0 $ with  weight $ w(v_0)=(a_0,b_0;m) $, one of $ a_0 $ or $ b_0 $ is zero.  Then, 
\[
H\uZ\smas X(\T)_+\simeq
			H\uZ\bigwedg H\uZ \smas S^{\lambda +2}\underset{[\mu]\in \T_d/\GG}{\bigvee} { H\uZ \smas \GG/C_d}_+\smas {S^{\lambda^{a_\mu-b_\mu}}} .
		\]

\end{thma}

In the above, the notation $\T_0$ stands for the $C_m$-fixed points of $\T$, and $\T_d$ refers to the vertices whose stabilizer is $C_d$ for $d\mid m$. The results in Theorem A depend on a hypothesis on the weights at vertices which are fixed under the $G$-action. We further prove that these hypotheses may be removed when the group is of prime power order. (See Theorem \ref{thm:cp-case} and Theorem \ref{cpnthm}) 
\begin{thmb}
a) Let $\T $ be an admissible weighted  $C_p$-equivariant tree such that $ p\nmid a_0, b_0$ but $ p\mid a_v-b_v $ for some fixed vertex $ v $.  	Then,  
		\begin{equation*}
				H\uZ\smas X(\T)_+ \simeq 
				H\uZ\bigwedg H\uZ \smas S^{{\lambda}+ 2}\bigwedg\limits_{\phi(\T)+1} H\uZ \smas S^{\lambda}\bigwedg\limits_{\psi(\T)-1}H\uZ \smas {S^2}\bigwedg\limits_{ \T_e/\GG} { H\uZ \smas \GG/e}_+\smas {S^2}. 
		\end{equation*}
b) Let  $\T $ be an admissible weighted  $C_{p^n} $-equivariant  tree.  Suppose $ \tau>0 $ is the maximum power of $ p $ that divides $ a_v-b_v $ among the fixed vertices $v$ and  $ p\nmid a_0, b_0 $. Then  
		\begin{equation*}
			H\uZ\smas X(\T)_+\simeq	H\uZ\bigwedg H\uZ \smas S^{\lambda+\lambda^{p^{\tau}}}\bigwedg\limits_{W_{\T}(i)} H\uZ \smas S^{\lambda^{i}}\bigwedg\limits_{{[\mu]\in \T_d/\GG,~d\neq p^n}} H\uZ \smas {\GG/C_d}_+\smas S^{\lambda^{a_\mu-b_\mu}}. 
		\end{equation*}		
\end{thmb}
In the statements of Theorem B, we observe that the complementary cases are proved in Theorem A. The notations $\phi(\T)$, $\psi(\T)$ and $W_\T(i)$ are clarified later in the document. The techniques used in the proof are the cellular filtration of the manifolds $X(\T)$, and the following result about the $RO(C_m)$-graded homotopy groups of $H\uZ$. (see Theorem \ref{thm:BG-C_m}) 	
\begin{thmc}
		Let $ \alpha\in RO(C_m) $ be such that $ |\alpha| $ is odd, and $|\alpha^H|>-1$ implies $|\alpha^K|\ge -1 $ for all subgroups $ K\supseteq H $. Then $ \upi_{\alpha}^{\GG}(H\uZ) =0$.
\end{thmc}
	
	\begin{mysubsection}{Organization}
In Section \ref{sec:2}, we recall the equivariant connected sum construction leading to the definition of tree manifolds associated to admissible weighted trees in \cite{HT04}. In Section \ref{sec:3}, we introduce the facts from equivariant stable homotopy theory required in the subsequent sections and prove Theorem C. In Section \ref{sec:4}, we prove the homology decompositions over the group $C_m$ (Theorem A), and finally in Section \ref{sec:5}, we prove the results over prime power groups (Theorem B). 
	\end{mysubsection}
	
	\begin{notation}
		Throughout this paper, $\GG$ denotes the cyclic group of order $m$, where $m$ is odd, and $g$ denotes a fixed generator of $\GG$. Whenever the notation $ p $ is used for a prime, it is supposed to be odd. For an orthogonal $\GG$-representation $V$, $S(V)$ denotes the unit sphere, $D(V)$ the unit disk, and $S^V$ the one-point compactification $\cong D(V)/S(V)$. The notation $ 1_\C $ is used for the trivial complex representation and $ 1 $ for the real trivial representation. 
	\end{notation}	
	
	\begin{ack}
The first author would like to thank Surojit Ghosh for some helpful conversations in the proof of Theorem \ref{thm:BG-C_m}. The research of the first author was supported by the SERB MATRICS grant 2018/000845. The research of the second author was supported by the NBHM grant no. 16(21)/2020/11. 
\end{ack}
	
	\section{Tree manifolds}\label{sec:2}
	In this section we discuss the construction of connected sum of $G$-manifolds focussing on the special case of a connected sum of complex projective planes in the case $G=\GG$. In the latter case, the construction is governed through a system of explicit combinatorial data expressed as admissible weighted trees (see \cite{HT04} for details).  We refer to these as {\it tree manifolds}.

	\begin{mysubsection}{Equivariant connected sums}	
		Let $X$ and $Y$ be two smooth $G$-manifolds of the same dimension $n$. The equivariant connected sum $X\# Y$ depends on the following data \\
		1) Points $x\in X^G$, $y\in Y^G$.\\
		2) An orientation reversing isomorphism of real $G$-representations $\varphi : T_x X \to T_y Y$. \\
		Given the data above, one may conjugate $\varphi$ with the exponential map to obtain a diffeomorphism of punctured disks near $x$ and $y$. This identification is then performed on $X\ssm \{x\} \sqcup Y \ssm \{y\}$ to obtain the equivariant connected sum $X\# Y$. One readily observes the following homotopy cofibration sequences 
		\begin{myeq} \label{conn-sum-cof}
			X\ssm \{x \} \to X\# Y \to Y, \text{ and}~~~ Y\ssm \{y\} \to X\# Y \to X.
		\end{myeq}
		An additional feature in the $G$-equivariant situation is the orbit-wise connected sum. Let $X$ be a $G$-manifold and $Y$ an $N$-manifold for a subgroup $N$. The data underlying an orbit-wise connected sum is \\
		1) A point $y\in Y^N$, and a point $x\in X$ such that the stabilizer of $x$ is $N$. \\
		2)  An orientation reversing isomorphism of real $N$-representations $\varphi : T_x X \to T_y Y$. 	\\
		The condition 1) implies that $x$ induces the inclusion of an orbit $i_x : G/N \to X$. Now we may again use the exponential map to conjugate $\varphi$ and identify punctured disks at points of $G/N \hookrightarrow X$ with those at points of  $G/N\times \{y\} \hookrightarrow G\times_N Y$. The resulting connected sum is denoted by $X\# G\times_N Y$. The direct analogues of \eqref{conn-sum-cof} are 
		\[
		X\ssm \{x \} \to X\# G\times_N Y \to \frac{G\times_N Y}{G/N \times \{y\}},
		\]
		and,
		\[
		G\times_N (Y\ssm \{y\}) \to X\#G\times_N Y \to X/i_x(G/N).
		\]
		The second sequence has a refinement in the form of a homotopy pushout 
		\begin{myeq}\label{orb-conn-sum-cof}
			\xymatrix{ G\times_N(Y-\{y\}) \ar[r] \ar[d] & G\times_N C(Y-\{y\}) \ar[d]^{i_x \circ \pi_1} \\ 
				X\# G\times_N Y \ar[r] & X. }
		\end{myeq}
	\end{mysubsection}	
	
	\begin{mysubsection}{Linear actions on projective spaces} 
		The principal construction of interest in this paper is the equivariant connected sum of projective spaces. A method to construct a $G$-action on a complex projective space $\C P^n$  is to write it as $P(V)$, the projectivization of a unitary representation $V$. We call these {\it linear actions}. If $\nu$ is a $1$-dimensional complex representation of $G$, there is an equivariant homeomorphism $P(V)\cong P(V\otimes \nu)$. In the case $G=\GG$, we fix the following notation for its representations.
		
		\begin{notation}	
			The irreducible complex representations of $\GG$ are $1$-dimensional, and up to isomorphism are listed as $1_\C,\lambda, \lambda^2,\dots, \lambda^{m-1}$ where $\lambda$ sends $g$ to  $e^{2\pi i/m}$, the $m^{th}$ root of unity. The real irreducible representations are the realizations of these. The realization of $\lambda^i$ is also denoted by the same notation. Note that $\lambda^i$ and $\lambda^{m-i}$ are conjugate and hence their realizations are isomorphic by the natural $\R$-linear map $z\mapsto \bar{z}$ which reverses orientation.   
		\end{notation}
		In this paper, our principal objects of interest are linear $\GG$-actions on $\C P^2$, that is, we write $\C P^2$ as $P(V)$ where $V$ is a $3$-dimensional complex representation of $\GG$. In terms of the notation above, $V$ is a sum $ \lambda^a + \lambda^b + \lambda^c$ for some integers $a,b$ and $c$ viewed $\pmod{m}$. As $P(V) \cong P(V\otimes\nu)$ for $1$-dimensional $\nu$,  we may assume $c=0$ in the expression for $V$. We denote this by $\C P^2(a,b;m)$. 
		\begin{prop}\label{prop-act}
			The manifolds $\C P^2(a,b;m)$ satisfy the following properties.\\ 
			1) If $\gcd(a,b,m) = d$, then $\C P^2(a,b;m)\cong \pi^*\C P^2(\frac{a}{d},\frac{b}{d};\frac{m}{d})$  the pullback via $\pi: C_m \rightarrow C_m/C_d \cong C_{m/d}$.\\
			2) There are $\GG$-homeomorphisms $\C P^2(a,b;m) \cong \C P^2(a-b,-b;m) \cong \C P^2(-a,b-a;m)$, and $\C P^2(a,b;m)\cong \C P^2(b,a;m)$. \\
			3) The points $p_1=[1,0,0]$, $p_2=[0,1,0]$, and $p_3=[0,0,1]$ are fixed by $\GG$. Their tangential representations are given by 
			\[
			T_{p_1}\C P^2(a,b;m) \cong \lambda^{b-a} + \lambda^{-a}, 	~ T_{p_2}\C P^2(a,b;m) \cong \lambda^{a-b} + \lambda^{-b},~T_{p_3}\C P^2(a,b;m) \cong \lambda^{a} + \lambda^{b}.
			\]
		\end{prop}
		The proof of the above easily follows from the homeomorphism $P(V)\cong P(V\otimes \nu)$, and the identification of the tangent bundle of $\C P^2$ as $\Hom(\gamma, \gamma^\perp)$, where $\gamma$ is the canonical line bundle. In this paper, we  call the numbers $(a,b)$ associated to the representation $\lambda^a+\lambda^b$ {\it rotation numbers}. As in \cite{HT04}, we assume that for the manifold denoted by $\C P^2(a,b;m)$, $\gcd(a,b,m)=1$.  
		
		We also denote $S^4(a,b;m)$ for the $\GG$-action on $S^4$ by identifying it with $S^{\lambda^a+\lambda^b}$. This may also be described as $S(1+\lambda^a+\lambda^b)$, where $1$ is the trivial real representation of dimension $1$. This action has fixed points $0$ and $\infty$, and the tangential representations are $\lambda^{a}+\lambda^{b}$ and $\lambda^{-a} + \lambda^{-b}$ respectively.
		
		We now list the conditions required to form equivariant connected sums of copies of $\C P^2(a,b;m)$ and $S^4(a,b;m)$. 
		\begin{prop}\label{cond-conn-sum}
			1) The connected sum $\C P^2(a,b;m) \# \C P^2(a',b';m)$ may be formed if and only if for one of the equivalent choices of $(a',b')$ as in 2) of Proposition \ref{prop-act}, $\pm (a',b')\in \{(a,-b), (a-b,b), (a,b-a)\}$. Once this condition is satisfied, there is a natural choice of data for the connected sum unless $a=b$ or one of $a$, $b$ is $0$.\\
			2) The connected sum $\C P^2(a,b;m) \# S^4(a',b';m)$ may be formed if and only if $\pm (a',b')\in \{(a,-b), (a-b,b), (a,b-a)\}$. In this case, $\C P^2(a,b;m)\# S^4(a',b';m)$ is $\GG$-homeomorphic to $\C P^2(a,b;m)$. \\
			3) For $m'\mid m$ but $m'\neq m$, the connected sum $\C P^2(a,b;m) \# \GG\times_{C_{m'}} \C P^2(a', b';m')$ may be defined if and only if the following are satisfied \\ 
a) One of $a$, $b$, or $a-b$ satisfies the equation $\gcd(x,m)=m'$. \\
b) One of $a'$, $b'$, or $a'-b'$ is $0$, and the others are up to sign two of the numbers in $\{-b,-a, b-a\}$ not divisible by $m'$. 
		\end{prop}
		The statements 1) and 2) follow from the examination of tangential representations at fixed points of $\C P^2(a,b;m)$, and 3) follows from the results in \cite[\S 1.C]{HT04}. 
	\end{mysubsection}

	\begin{mysubsection} {Admissible weighted trees} The observations above inform us that the $\GG$-connected sums of different $\C P^2(a,b;m)$s and $S^4(a,b;m)$s may be formed only when certain relations are satisfied between the weights involved.  We now lay down the sequence of combinatorial criteria which allow us to form such a connected sum. These are written in the form of weights attached to trees with a $\GG$-action satisfying required conditions, called {\it admissible weighted trees}.
		
		Recall that a group action on a tree is given by an action on the vertices which preserves the adjacency relation. We define two types of trees called {\it Type I} and {\it Type II}.
		\begin{defn}	\label{adm-wt-tree}
			An {\it admissible weighted tree} is a  tree with $\GG$-action having the following properties 
			\begin{enumerate}
				\item There is a $\GG$-fixed vertex $v_0$ called the root vertex of the tree. In case of a type II tree, $v_0$ is the unique $\GG$-fixed vertex. 
				\item The vertices of $\mathbb{T}$ are arranged in levels starting from zero according to the distance from the root vertex with edge length considered to be $1$. Observe that $\GG$ preserves the levels and every edge goes from level $L$ to $L+1$ for some $L$.
				\item Each vertex $v$ is equipped with a weight $w(v)=(a_v,b_v;m_v)$ (defined up to equivalence $(a_v,b_v;m_v)\sim (b_v,a_v;m_v)$) such that $m_v \mid m$ and $m_{v_0}=m=|\GG|$,   and $\gcd (a_v,b_v,m_v) = 1$ for {all $v$}.
				\item For every vertex $v$, $Stab(v) = C_{m_v}\subset \GG$. Also $w(g.v) = w(v)$, so that weights of vertices in the same orbit are equal. 
				\item In the case of type I trees, there are at most three vertices $v$ of level $1$ such that $m_v = m$. Each of these vertices have distinct weights (up to equivalence) among $\{\pm (a_{v_0},-b_{v_0}), \pm (a_{v_0},b_{v_0}-a_{v_0}), \pm (b_{v_0},a_{v_0}-b_{v_0})\}$.
				\item Vertices with the same weight (up to equivalence and sign) do not have a common neighbour unless they are related by the $\GG$-action. 
				\item Suppose there is an edge $e$ from $ v$ in level $L$ to $u$ in level $L+1$. Then $m_{u}\mid m_v$ and 
				\begin{enumerate}
					\item If $m_{u}= m_v$  and $ v $ is not the root vertex, then $ \pm (a_u,b_u) \in \{(a_v,b_v-a_v),(a_v-b_v,b_v)\}$.
					\item If $m_{u} \neq m_v$ then one of $a_v, b_v, a_v-b_v$ satisfies the equation $\gcd(x,m)=m_u$, accordingly $b_{u} = 0$, and $a_{u} $ equals  a value among $\pm$ $\{-a_v, -b_v, b_v-a_v\}$ not divisible by $m_{u}$.
				\end{enumerate}
				
			\end{enumerate}
		\end{defn}
		
		\begin{rmk}
			The definition of admissible, weighted tree above is the same definition as \cite[\S 1.D]{HT04}. To see this, one may observe the following
			\begin{itemize}
				\item The tree as defined inherits a direction, where an edge $e$ moving from level $L$ to $L+1$ is directed so that $\partial_0 e$ lies in level $L$, and $\partial_1 e$ lies in level $L+1$.  One also observes that a vertex in level $L>0$ is connected to a unique vertex in level $L-1$. 
				\item The partial order may be generated from the condition that $\partial_0 e < \partial_1 e$. This implies that two vertices are comparable if they are connected by a sequence of edges, and in this case the order relation is determined by the level. 
				\item The weights $w(v)=(a_v,b_v;m_v)$ are so defined that we obtain an equivalent weight under the operations $(a_v,b_v) \mapsto (b_v, a_v)$. This is the equivalence of weights referred to in the definition above. 
				\item The conditions (5) and (7) above reflect the condition ``pair of matching fixed components'' of \cite{HT04}. As we shall see, this is a slightly stronger condition that also includes the data required for us to form the corresponding equivariant connected sum. 
			\end{itemize}
		\end{rmk}
		
		\begin{notation}
			The number $n(\T)$ associated to an admissible, weighted tree $\T$ with vertex set $V(\T)$ is defined as 
			\[
			n(\T)= \begin{cases} \#(V(\T)) & \mbox{ if } \T \mbox{ is of type I} \\ 
				\#(V(\T)) - 1 & \mbox{ if } \T \mbox{ if of type II}. \end{cases} 
			\] 
		\end{notation}		
		
		The significance of the notation $n(\T)$ is that we associate to an admissible weighted tree $\T$, a $\GG$-manifold $X(\T)$ whose underlying space is $\#^{n(\T)} \C P^2$. 	We will use the notation $ \T_0 $ for $ \T^{C_p} $, and $ \T_d=\{v\in \T\mid Stab(v)=C_d\} $. Observe that  $ \GG/C_d$ acts freely on $ \T_d $.
	\end{mysubsection}
	
	\begin{mysubsection}{Construction of connected sums along trees} The construction of $X(\T)$, the $\GG$-manifold obtained by the connected sum of linear actions according to the data described in the tree $\T$ is carried out in \cite[Theorem 1.7]{HT04}. We describe it's main features below. For a vertex $v\in V(\T)$, we use the notation 
		\[
		\C P^2_v := \C P^2 (a_v,b_v;m_v).
		\]
		\begin{prop} \label{tree-mfld-const}
			Given an admissible, weighted tree $\T$, there is a $\GG$-manifold $X(\T)$ such that 
			\begin{enumerate}
				\item The underlying space of $X(\T)$ is $\#^{n(\T)} \C P^2$. 
				\item If $\T$ is of type I, then $X(\T) \cong$ a connected sum of copies of $\C P^2_v$ for every vertex $v$ of $\T$. 
				\item If $\T$ is of type II, then $X(\T) \cong$ a connected sum of copies of $\C P^2_v$ for every non-root vertex $v$ of $\T$, and a copy of $S^4(a_{v_0},b_{v_0};m)$. 
				\item For a non-root vertex $v$ in level $L$, which is connected to $w$ in level $L-1$ with $m_v=m_w$, the points where the connected sum is performed are $[0,0,1] \in \C P^2_v$,  and the one in $\C P^2_w$ determined by the condition (7)(a) of Definition \ref{adm-wt-tree} if $w$ is not the root vertex, or by (5) of Definition \ref{adm-wt-tree} if $w=v_0$. 
				\item For a non-root vertex $v$ in level $L$ connected to $w$ in level $L-1$ with $m_v<m_w$, the points where the connected sum is performed are $[0,0,1] \in \C P^2_v$,  and some equivalent choice of point in $\C P^2_w$ determined by the condition (7) (b) of Definition \ref{adm-wt-tree}. Equivalent choices of the latter give equivalent manifolds \cite[Lemma 1.2]{HT04}. 
			\end{enumerate}
		\end{prop}
		
		We now elaborate further on (4) and (5) of Proposition \ref{tree-mfld-const} above. We start with an example. 
		\begin{exam}\label{conn-sum-cp2}
			In order to see if $\C P^2(a,b;m)\# \C P^2(a',b';m)$ is definable we may apply 1) of Proposition \ref{cond-conn-sum}. Another method of saying this is that there is an expression of the second summand as $\C P^2(a',b';m)$ such that $\pm (a',b') \in \{(a,-b), (a-b,b), (a,b-a)\}$. Once this choice is made, say $(a',b')=(a,-b)$, we get a natural data for the equivariant connected sum as 
			\begin{enumerate}
				\item The point  $p\in \C P^2(a,b;m)$ used in the connected sum is $[0,0,1]$, and the corresponding tangential representation is $\lambda^a+\lambda^b$. 
				\item The point $q\in \C P^2(a',b';m)$ used in the connected sum is $[0,0,1]$, and the corresponding tangential representation is $\lambda^{a'} + \lambda^{b'}=\lambda^a + \lambda^{-b}$. 
				\item The natural orientation reversing isomorphism $T_p \C P^2(a,b;m) \to T_q \C P^2(a',b';m)$ is given by identity on the factor $\lambda^a$ and complex conjugation on the factor $\lambda^b$. 
			\end{enumerate} 
		\end{exam}
		
\vspace*{0.4cm}

		In (4) of Proposition \ref{tree-mfld-const}, the choice of $(a_v, b_v)$ implies that 
		\[
		T_{[0,0,1]}\C P^2_v =  \lambda^{a_v} + \lambda^{b_v}   
		\]
		equals one of $\lambda^{a_w} +\lambda^{b_w-a_w}$, $\lambda^{-a_w} +\lambda^{a_w-b_w}$, $\lambda^{b_w} +\lambda^{a_w-b_w}$, $\lambda^{-b_w} +\lambda^{b_w-a_w}$, in the case $w$ is not the root vertex. We also note  the tangential representations 
		\[
		T_{[1,0,0]}\C P^2_w = \lambda^{-a_w} + \lambda^{b_w-a_w}, ~	T_{[0,1,0]}\C P^2_w = \lambda^{-b_w} + \lambda^{a_w-b_w}.
		\]
		Among the possibilities for $T_{[0,0,1]}\C P^2_v$, the first two are compatible with $T_{[1,0,0]}\C P^2_w$, and the second two are compatible with $T_{[0,1,0]}\C P^2_w$. This demonstrates how the weights imply the choice of connected sum point in $\C P^2_w$. The argument in Example \ref{conn-sum-cp2} applies here to construct a canonical orientation reversing isomorphism among the tangential representations. Finally the condition (6) of Definition \ref{adm-wt-tree} implies that the choice of connected sum point is not the same as that of any other vertex. 
		
		We now look at (5) of Proposition \ref{tree-mfld-const}. The condition $m_v<m_w$ implies that $m_v$ is a proper divisor of $m_w$. Consider $\OO\cong C_{m_w}/C_{m_v}$, the orbit of $v$ under the $C_{m_w}$-action. From (4) of Definition \ref{adm-wt-tree}, we observe that all the vertices in $\OO$ have weight $w(v)$. The connected sum formed here is $\C P^2_w \# C_{m_w}\times_{C_{m_v}} \C P^2_v$, which connects the manifolds at all the vertices in $\OO$ to $\C P^2_w$ at one go by writing 
		\[
		C_{m_w}\times_{C_{m_v}} \C P^2 \cong \coprod_{h\in C_{m_w}/C_{m_v}} h \cdot \C P^2_v =  \coprod_{h\in C_{m_w}/C_{m_v}}  \C P^2_{h\cdot v}.
		\]
		In this case, we have $b_v=0$ and $m_v$ divides one of the numbers $a_w$, $b_w$, $b_w-a_w$ but not more than one (unless $m_v=1$) as $\gcd(a_w,b_w,m_w)=1$. We may assume $\gcd( b_w,m_w)=m_v$ without loss of generality, and it implies $a_v=\pm a_w$. The first part of the equivariant data for the connected sum is the point $[0,0,1]$ in $\C P^2_v$ with tangential $C_{m_v}$-representation $\lambda^{a_v}+1_\C$ (as a complex representation). The next part is a choice of connected sum point which is required to have stabilizer $C_{m_v}$, and hence belongs to 
		\[
		P(\lambda^{b_w}+1_\C) - \{[0,1,0], [0,0,1]\} \subset P(\lambda^{a_w}+\lambda^{b_w}+1_\C)=\C P^2_w. 
		\]
		For any point  $q\in P(\lambda^{b_w}+1_\C) - \{[0,1,0], [0,0,1]\}$ the tangential $C_{m_v}$-representation is $1_\C+\lambda^{a_w}$. We now have a canonical orientation reversing isomorphism between $T_{[0,0,1]}\C P^2_v$ and $T_q \C P^2_w$ which is conjugation on $\lambda^{a_v}$ if $a_v=-a_w$, or conjugation on the other factor if $a_v=a_w$. As $P(\lambda^{b_w}+1_\C) - \{[0,1,0], [0,0,1]\}$ is connected, this defines the equivariant connected sum up to diffeomorphism (\cite[Lemma 1.2]{HT04}). Note also that there is a completely analogous version of the above if $w$ was the root vertex of a type II tree, and $\C P^2_w$ was replaced by $S^{\lambda^{a_w}+\lambda^{b_w}}$. 
		
	\end{mysubsection}

	\section{Equivariant homology  for cyclic groups} \label{sec:3}
	In this section, we recall the definition of equivariant homology with coefficients in a Mackey functor. The main objective is to describe a theorem on $\uZ$-homology which allows us to construct the homology decompositions in the following sections. Equivariant homology and cohomology possess the richest structure when the coefficients are Mackey functors \cite{Dre72}, which we summarize in explicit terms below. 
	\begin{defn}
		A $\GG$-Mackey functor\footnote{This is a simplification in the case $\GG$ is Abelian. Otherwise the double coset formula (4) has a slightly more complicated expression.} $\uM$ is a collection of commutative $\GG/H$-groups $\uM(\GG/H)$ for each subgroup $H \le \GG$, accompanied by \emph{transfer} $\tr^H_K\colon  \uM(\GG/K) \to \uM(\GG/H)$ and \emph{restriction} $\res^H_K \colon\uM(\GG/H) \to \uM(\GG/K)$ for $K\le H\le \GG$ such that
		\begin{enumerate}
			\item $\tr^H_J = \tr^H_K\tr^K_J$ and $\res^H_J = \res^K_J \res^H_K$ for all $J \le K \le H.$
			\item  $\tr^H_K(\gamma.x)= \tr^H_K(x)$ for all $x \in \uM(\GG/K)$ and $\gamma \in H/K.$
			\item $\gamma. \res^H_K(x) =\res^H_K(x)$ for all $x \in \uM(\GG/H)$ and $\gamma \in H/K.$
			\item $\res^H_K\tr ^J_K(x)= \sum_{\gamma \in H/K} \gamma.\tr^{K}_{J\cap K}(x)$
			for all subgroups $J,H \leq K.$
		\end{enumerate}
	\end{defn}

\begin{exam}	
The Burnside ring Mackey functor denoted $\uA$ is described by $\uA(\GG/H)=A(H)$, the Burnside ring of $H$. This is the group completion of the monoid of finite $H$-sets up to isomorphism. The restriction maps are given by restricting the action, and the transfer maps are given by inducing up the action : $S \mapsto H\times_K S$ for $K\leq H$. 

In this paper, we work primarily with the constant Mackey functor $\uZ$  described by 
	\[
	\uZ(\GG/H)= \Z,~~ \res^H_K=Id,~~ \tr^H_K = [H:K],  	
	\]
	for $K\leq H \leq \GG$. One may make a dual construction to define the Mackey functor $\uZ^\ast$ by 
\[
	\uZ^\ast(\GG/H)= \Z,~~ \res^H_K=[H:K],~~ \tr^H_K = Id,  	
	\]
	for $K\leq H \leq \GG$. For an Abelian group $C$, the Mackey functor $\langle C \rangle$ is described by
 \[
	\langle C \rangle(\GG/H)= \begin{cases} C & \mbox{ if } H=\GG \\ 0 & \mbox{ otherwise}. \end{cases}  	
	\]
\end{exam}

\vspace*{0.4 cm}

The importance of Mackey functors from the point of view of ordinary cohomology in the equivariant case is due to the following result. 
	\begin{thm}\cite[Theorem 5.3]{GM95} 
		For a Mackey functor $\uM$, there is an Eilenberg-MacLane $\GG$-spectrum $H\uM$ which is unique up to isomorphism in the equivariant stable homotopy category. 
	\end{thm}
	This Eilenberg-MacLane spectra are those whose integer-graded homotopy groups are concentrated in degree $0$ in the category of equivariant orthogonal spectra \cite{MM02}. The homotopy category of equivariant orthogonal spectra is called the equivariant stable homotopy category, where one has desuspension functors for one-point compactifications of orthogonal $\GG$-representations. 
	\begin{exam}
		For a $\GG$-spectrum $X$, the equivariant homotopy groups possess the structure of a Mackey functor $\upi_n(X)$, defined by the formula 
		\[ \upi_n(X)(\GG/H):= \pi_n(X^H).
		\]
		Analogously the cohomology theory and homology theory associated to Mackey functors are $RO(\GG)$-graded and may also be equipped with the structure of a Mackey functor which on objects is described as 
		\[
		\uH^\alpha_\GG(X;\uM)(\GG/K) \cong \mbox{ Ho-}\GG\mbox{-spectra } (X\smas \GG/K_+ , \Sigma^\alpha H\uM),
		\]
		\[
		\uH_\alpha^\GG(X;\uM)(\GG/K) \cong \mbox{ Ho-}\GG\mbox{-spectra } (S^\alpha \smas \GG/K_+ , X \smas H\uM),
		\]
		for $\alpha\in RO(\GG)$. 
	\end{exam}
	
\vspace*{0.4cm}

	The Mackey functor $\uZ$ has a multiplicative structure which makes it a {\it commutative Green functor} \cite[Chapter XIII.5]{May96}. The consequence of this multiplication is that the cohomology $H^\bigstar_\GG(X;\uZ)$ has a graded commutative ring structure. The multiplicative structure also allows us to consider the Mackey functors which are $\uZ$-modules, and examples of these are the homology and the cohomology Mackey functors $\uH_\alpha^\GG(X;\uZ)$ and $\uH^\alpha_\GG(X;\uZ)$. 
	\begin{rmk}\label{cohomological}
		For  any $\uM\in \uZ$-$\Mod_\GG$, $\tr^H_K \res^H_K$ equals the multiplication by index $[H: K]$ for $K \le H \le \GG$ \cite[Theorem 4.3]{Yos83}.
	\end{rmk}
	
\vspace*{0.4cm}
	The spectrum $H\uZ$ also has the following well-known relation after smashing with representation spheres. 
	\begin{prop}\label{smash-rep}	
		If $ (d,m)=1 $, then 
		\[
		H\uZ\smas S^{\lambda^k} \simeq H\uZ\smas S^{\lambda^{dk}}.
		\]
	\end{prop}	
	
	\begin{proof}
		It suffices to assume $k\mid m$,  so that $g^k$ generates the subgroup isomorphic $C_{m/k}$ of $\GG$. The spaces $S^{\lambda^k}$ and $S^{\lambda^{dk}}$ fit into cofibre sequences of spectra as follows 
		\[
		{\GG/C_k}_+ \stackrel{1-g^k}{\to} {\GG/C_k}_+ \to S(\lambda^k)_+, ~~ S(\lambda^k)_+\to S^0 \to S^{\lambda^k},
		\]
		\[
		{\GG/C_k}_+ \stackrel{1-g^{dk}}{\to} {\GG/C_k}_+ \to S(\lambda^{dk})_+, ~~ S(\lambda^{dk})_+\to S^0 \to S^{\lambda^{dk}}.
		\]
		The crucial observation here is that as maps $H\uZ \wedge {\GG/C_k}_+ \to H\uZ \wedge {\GG/C_k}_+$, $H\uZ \wedge g^k \simeq H\uZ \wedge g^{dk}$. It follows that $H\uZ \wedge S(\lambda^k)_+ \simeq H\uZ \wedge S(\lambda^{dk})_+$, and hence, $H\uZ \wedge S^{\lambda^k} \simeq H\uZ \wedge S^{\lambda^{dk}}$. 
	\end{proof}	
	
	We may observe from Proposition \ref{smash-rep} that $S^{\lambda^k-\lambda^{dk}} \wedge H\uZ \simeq H\uZ$. This means in the graded commutative ring $\pi_\bigstar H\uZ$ (graded over $RO(\GG)$), there are invertible classes in degrees $\lambda^k - \lambda^{dk}$ whenever $(d,m)=1$. As a consequence the ring $\pi_\bigstar H\uZ$ is determined from it's values at the gradings which are linear combinations of $\lambda^k$ for $k\mid m$. We recall the following computation of $\upi_\bigstar^{C_p}(H\uZ)$ \cite[Appendix B]{Fer99}.
\begin{myeq}\label{cp-comp}
 \upi_{\alpha}^{C_p}(H \uZ) = \begin{cases} \uZ & \text{if} \; |\alpha| =0, \; |\alpha^{C_p}| \geq 0 \\
 \uZ^* & \text{if} \; |\alpha| =0, \; |\alpha^{C_p}| < 0 \\
 \langle \Z/p \rangle & \text{if}\;  |\alpha| <0, \; |\alpha^{C_p}| \geq 0, \; \text{and} \; |\alpha| \; \text{even} \\ 
\langle \Z/p \rangle & \text{if} \;|\alpha|>0, \; |\alpha^{C_p}| <-1,\; \text{and} \; |\alpha| \; \text{odd} \\ 
0 & \text{otherwise.} \end{cases}
\end{myeq}

	The homology decomposition theorems for $\GG$-spaces with even cells require proving that certain odd degree homotopy groups of $H\uZ$ are $0$. In this paper, we use the following result. 
	\begin{thm}\label{thm:BG-C_m}
		Let $ \alpha\in RO(\GG) $ be such that $ |\alpha| $ is odd, and $|\alpha^H|>-1$ implies $|\alpha^K|\ge -1 $ for all subgroups $ K\supseteq H $. Then $ \upi_{\alpha}^{\GG}(H\uZ) =0$.
	\end{thm}	
	
	\begin{proof}
		The proof relies on \cite[Proposition 4.3]{BG21}, where the same result is proved for the groups $C_{p^n}$ where $p$ is an odd prime. We now suppose that the result is true for all subgroups of $\GG$ and then prove it for $\GG$. In this way, the result will be true for all cyclic groups of odd order. 
		
		For an $\alpha$ satisfying the hypothesis of the theorem, let $x\in \pi_\alpha^\GG(H\uZ)$. As the result has already been proved at prime powers we may assume that $m$ is divisible by at least two distinct primes $p$ and $q$. The result is also true for $C_{m/p}$ and $C_{m/q}$ by the hypothesis. We then compute
		\[
		px= [\GG: C_{m/p}] x= \tr^\GG_{C_{m/p}} \res^\GG_{C_{m/p}} (x) = 0, 
		\]
		as $\pi_\alpha^{C_{m/p}}(H\uZ)=0$, and analogously, 
		\[
		qx= [\GG: C_{m/q}] x= \tr^\GG_{C_{m/q}} \res^\GG_{C_{m/q}} (x) = 0. 
		\]
		It follows that $x=0$. 
	\end{proof}
	
	Theorem \ref{thm:BG-C_m} is useful to prove that cohomology of $\GG$-spaces which are constructed by attaching even cells of the type $D(V)$,   is a free module over the cohomology of a point $\uH^\bigstar_\GG(S^0;\uZ)$. Such results have been proved in \cite{Lew88}, \cite{BG19}, \cite{BG20}, in the context of equivariant projective spaces and Grassmannians. A more careful argument has also been used in \cite{Fer99} and \cite{FL04}, where the free module property has been proved for all finite complexes obtained by attaching cells of the type $D(V)$ in even dimensions.  
	
	\section{Equivariant homology decompositions  for tree manifolds } \label{sec:4}
	In this section, we obtain homology decompositions for the tree manifolds defined in Section \ref{sec:2}. Recall that, $ \C P^2(a,b;m) $ serves as a building block for these manifolds. We describe a cellular decomposition of complex projective spaces, which has been studied along with cohomology of such spaces   in \cite[\S 3]{Lew88} and \cite[\S 8.1]{BG19}.
	
	\begin{mysubsection}{Cellular filtration of projective spaces}
		The equivariant complex projective space $ P(V) $ is built up by attaching even dimensional cells of the type $ D(W) $ for real representations $W$. To see this, let $ V_n $ be a  $ \GG $-representation  that decomposes in terms of irreducible factors as $ V_n=\sum_{i=0}^n \phi_i $, and  let $ W_n $ denote the $ \GG $-representation $ \phi_n^{-1}\otimes  \sum_{i=0}^{n-1} \phi_i $. Consider the $ \GG $-equivariant map $ D(W_n)\to P(V_n)\cong P(\phi_n^{-1}\otimes V_{n}) $  defined by
		\[
		(z_0, z_1,\dots, z_{n-1})\mapsto [z_0, z_1,\dots, z_{n-1}, 1-\sum_{i=0}^{n-1}|z_i^2| ],
		\]
		where $ z_i\in \phi_n^{-1}\otimes \phi_i $. Restricting this map to $ S(W_n) $, we see that its image  lies in $ P(V_{n-1}) $ (which may be regarded as  a subspace of $ P(V_n) $ in the obvious way), and it is a homeomorphism from $ D(W_n)\setminus S(W_n) $ to $P(V_n)\ssm P(V_{n-1})$.
		Thus, $ P(V_n) $ is obtained from $ P(V_{n-1}) $ by attaching the cell $ D(W_n) $ along this boundary map.  
		Observe that this filtration depends on the choice of the ordering of the $ \phi_i $'s.\par

		Returning to our example $ \C P^2(a,b;m) = P(\lambda^a\oplus\lambda^b \oplus 1_\C)$, we see that there are six possible ways to build it. This choices will play a crucial role in proving the homology decomposition theorems, as we will see below.
		\begin{exam}\label{cp2homdec}
			Writing $\C P^2(a,b;m)= P(\lambda^a\oplus\lambda^b\oplus 1_\C) $ in this order, the cellular filtration above gives us the following cofibre sequence (using the fact that $P(\lambda^a\oplus\lambda^b) \cong S^{\lambda^{a-b}}$)
			\[
			S^{\lambda^{a-b}} \to \C P^2(a,b;m) \to S^{\lambda^a + \lambda^b}. 
			\]
			Using the other orderings, we also obtain the following cofibre sequences 
			\[
			S^{\lambda^{a}} \to \C P^2(a,b;m) \to S^{\lambda^{a-b} + \lambda^{-b}},~~S^{\lambda^b} \to \C P^2(a,b;m) \to S^{\lambda^{b-a} + \lambda^{-a}}. 
			\]
			The homology decomposition is obtained by smashing these cofibre sequences with $H\uZ$ and trying to prove a splitting. For example, in the cofibre sequence
			\begin{myeq}\label{cofseqcp2}
				H\uZ\wedge S^{\lambda^{a-b}} \to H\uZ\wedge  \C P^2(a,b;m) \to H\uZ \wedge S^{\lambda^a + \lambda^b}, 
			\end{myeq}
			the connecting map $H\uZ \wedge S^{\lambda^a + \lambda^b} \to H\uZ\wedge S^{\lambda^{a-b}+1}$ is a $H\uZ$-module map which is classified up to homotopy by $\pi_0^\GG (H\uZ \wedge S^{\lambda^{a-b}+1 - \lambda^a - \lambda^b})$. This group is now analyzed using Theorem \ref{thm:BG-C_m} at $\alpha = -\lambda^{a-b} - 1 + \lambda^a + \lambda^b$. Observe that $|\alpha|= 1 >0$, so in order to show $\pi_\alpha^\GG H\uZ=0$, we need $|\alpha^{C_d}|\geq -1$ for all $d \mid m$. Under the condition $\gcd(a,b,m)=1$, this is true if and only if $a-b$ is relatively prime to $m$, and in this case, 
			\[
			H\uZ \wedge \C P^2(a,b;m) \simeq H\uZ \wedge S^{\lambda^{a-b}} \bigvee H\uZ \wedge S^{\lambda^a+\lambda^b} .
			\]
			Using the other two cofibre sequences for $\C P^2(a,b;m)$, we see that a homology decomposition is obtained if one of $a$, $b$, or $a-b$ is relatively prime to $m$. 
		\end{exam}
	\end{mysubsection}
	
	\vspace*{0.4cm}	
	
	In the case of connected sums, we carry forward the homology decomposition argument of Example \ref{cp2homdec}. We illustrate this in the following example.
	\begin{exam}\label{exconnsum}
		Let  $X= \C P^2(a,b;m) \#  \C P^2(a',b';m)$ where $ \gcd(a'-b',m)=1 $ and $\gcd(a-b,m)=1$. We assume that the connected sum point $p$ in $\C P^2(a',b';m)$ has tangential representation $\lambda^{a'}\oplus\lambda^{b'}$ as in 1) of Proposition \ref{cond-conn-sum}. Example \ref{cp2homdec} shows that 
		\[
		H\uZ\smas  \C P^2(a,b;m)\simeq  H\uZ\smas  S^{\lambda^{a}+ \lambda^{b}}\bigwedg H\uZ\smas 	{S^\lambda}^{a-b}.
		\]
		
		To compute $H\uZ \wedge X$, we use the cofibre sequence \eqref{conn-sum-cof}. We note that	 
		\[
		\C P^2(a',b';m)\ssm \{p\}\simeq P(\lambda^{a'}\oplus \lambda^{b'})\simeq {S^{\lambda}}^{a'-b'}.
		\]
		Therefore we obtain a cofibre sequence of $H\uZ$-modules
		\begin{equation*}
			H\uZ \smas {S^\lambda}^{a'-b'} \to  H\uZ\wedge X  \to  H\uZ\smas  S^{\lambda^{a}+ \lambda^{b}}\bigwedg H\uZ\smas 	{S^\lambda}^{a-b}.
		\end{equation*}
		The hypothesis $\gcd (a'-b',m) =1$ implies using Theorem \ref{thm:BG-C_m} that the above  sequence splits. Consequently, we obtain
		\[
		H\uZ\smas X_+\simeq H\uZ\smas  H\uZ\smas  S^{\lambda^{a}+ \lambda^{b}}\bigwedg H\uZ\smas 	{S^\lambda}^{a-b}\bigwedg H\uZ\smas {S^\lambda}^{a'-b'}.
		\]
	\end{exam}

	\vspace*{0.4cm}
	
	We now prove the main theorems of this section. Example \ref{cp2homdec} points out the necessity of the hypothesis in the theorem. 
	\begin{thm}\label{thm:gen}
		If $ \T $  is an  admissible weighted tree of type I with  $ \GG $-action such that for all vertices $ v \in \T_0$ with  $w(v)= (a_v,b_v;m_v) $,  $\gcd(a_v-b_v,m_v)=1$, then, the $ H\uZ$-module  $ H\uZ\smas X(\T)_+ $ admits the decomposition  
\[
		H\uZ\smas X(\T)_+\simeq
			H\uZ\bigwedg H\uZ \smas S^{\lambda ^{a_0}+ \lambda^{b_0}} \underset{\T_0}{\bigvee} H\uZ\wedge S^\lambda \underset{[\mu]\in\T_d/\GG, d\neq m}{\bigvee} { H\uZ \smas \GG/C_d}_+\smas S^{\lambda^{a_\mu-b_\mu}} 
		\]
where $w(v_0)= (a_0,b_0;m) $. If $\T$ is of  type II, 
		\[
		H\uZ\smas X(\T)_+\simeq
			H\uZ\bigwedg H\uZ \smas S^{\lambda ^{a_0}+ \lambda^{b_0}}\underset{[\mu]\in\T_d/\GG,~ d\neq m}{\bigvee} { H\uZ \smas \GG/C_d}_+\smas {S^{\lambda^{a_\mu-b_\mu}}}. 
		\]	
	\end{thm}

	\begin{proof}
		We proceed by induction on $L(\T)$, the maximum level reached by vertices of the tree.
		The induction starts from a tree with only the root vertex. In the type I case, this is computed in Example \ref{cp2homdec}.  	
		In case of type II, the manifold is $  S^4(a_0,b_0;m) $, for which we have the following decomposition
		\[
		H\uZ \smas  S^4(a_0,b_0;m)_+\simeq H\uZ\bigwedg H\uZ \smas S^{\lambda ^{a_0}+ \lambda^{b_0}}.
		\]
		Assume that the statement holds whenever $L(\T)\leq L $. We prove it for trees with $L(\T)=L+1$. Given a tree $\T$ we denote by $ \T(L) $ the part of it up to level $L$, so that the result holds for $X(\T(L))$. We attach orbits  of the level  $ L+1 $ vertices one at a time. 
		We write down the argument for a type I tree, as the other case is entirely analogous.  Let $ \mathcal{O}_1,\dots, \mathcal{O}_k  $ denote the orbits of the  level $ L+1 $ vertices. It suffices to prove the case when an orbit $ \OO_i $ is added to $  \T(L)$ together with the attaching edges, which we denote by $ \T(L)+ \OO_i $. The  stabilizer for the vertices in $ \OO_i $ can be either the whole group  $ \GG $ or a smaller  subgroup $ C_d $. We deal these cases separately.\par
		{\it Case  1}:  The stabilizer for the vertices in $ \OO_i $ is $ \GG $, that is, $ \OO_i= \{v_i\}$. Suppose $w(v_i)= (a_i,b_i;m) $. By Proposition \ref{tree-mfld-const}, this implies that the tangential representation at the connected sum point of  $\C P^2_{ v_i} $ is $ {\lambda^{a_i}\oplus \lambda^{b_i}} $. We have the following cofibre sequence from \eqref{conn-sum-cof}
		\begin{equation*}
			\C P^2(a_i,b_i;m)\setminus D(\lambda^{a_i}\oplus \lambda^{b_i}) \to X\big(\T(L)+ \OO_i\big)\to X\big(\T(L)\big).
		\end{equation*}
		The left hand term  can be simplified further as 
		\[
		\C P^2(a_i,b_i;m)\setminus D(\lambda^{a_i}\oplus \lambda^{b_i})\simeq P(\lambda^{a_i}\oplus \lambda^{b_i})\simeq {S^{\lambda}}^{a_i-b_i}.
		\]
		We now apply Proposition \ref{smash-rep} to note that $ {H\uZ\smas  S^\lambda}^{a_i-b_i} \simeq { H\uZ\smas  S^\lambda}  $. Applying the induction hypothesis on $ X\big(\T(L)\big) $, we get a cofibre sequence of 
		$ H\uZ$-modules
		\begin{myeq}\label{eq:ind-gen-stab-G}
			{H\uZ\smas S^\lambda} \to  H\uZ\smas  X\big(\T(L)+ \OO_i\big)\to  H\uZ \smas S^{\lambda ^{a_0}+ \lambda^{b_0}}\bigwedg_{ \T(L)_d/\GG} { H\uZ \smas \GG/C_d}_+\smas {S^{\lambda}}.
		\end{myeq}
		Next we observe that the cofibre sequence splits by showing that up to homotopy, the connecting map  from each summand  of the right hand side of equation \eqref{eq:ind-gen-stab-G}  to $   H\uZ\smas  S^{\lambda+1} $ is zero. This follows from Theorem \ref{thm:BG-C_m}, and this gives the required homology decomposition for $X(\T(L) +\OO_i)$.\par
		{\it 	Case 2}: The vertices in $ \OO_i $ have  stabilizer $ C_{m_i} <\GG$ and   $w(v_i)= (a_i,b_i;m_i) $. In this case, we are considering the connected sum of the form $ X\big(\T(L)\big) \# \GG\times_{C_{m_i}} \C P^2(a_i,b_i;m_i) $. 	Consider the following homotopy pushout square of $ \GG $-spaces \eqref{orb-conn-sum-cof}
		\begin{equation*}
			\begin{tikzcd}
				\GG\times_{C_{m_i}} {S^\lambda}^{a_i-b_i} \arrow[r] \arrow[d]  &  \GG\times_{C_{m_i}} C({S^\lambda}^{a_i-b_i}) \arrow[d] \\
				X\big(\T(L)+ \OO_i\big) \arrow[r]           &           X\big(\T(L)\big),
			\end{tikzcd}
		\end{equation*}
		where $ C({S^\lambda}^{a_i-b_i}) $ denotes the cone of ${S^\lambda}^{a_i-b_i}  $. In $\GG$-spectra, this gives rise to the cofibre sequence
		\[	
		\GG/{C_{m_i}} _+\smas {S^\lambda}^{a_i-b_i}_+ \to  X\big(\T(L)+ \OO_i\big)_+ \bigwedg \GG/{C_{m_i}} _+ \to  X\big(\T(L)\big)_+.
		\]
		We use 
		\[
		\GG/{C_{m_i}} _+\smas {S^\lambda}^{a_i-b_i}_+ \simeq (\GG/{C_{m_i}}_+\smas {S^\lambda}^{a_i-b_i})\bigwedg \GG/{C_{m_i}} _+,
		\]
		to deduce the  cofibre sequence 
		\[ 
		\GG/{C_{m_i}} _+\smas {S^\lambda}^{a_i-b_i} \to  X\big(\T(L)+ \OO_i\big)_+  \to  X\big(\T(L)\big)_+.
		\]
	Now we take the smash product with $H\uZ$ to get the following cofibre sequence of $ H\uZ$-modules
		\begin{myeq}\label{eq:ind-gen-sm-stab}
			H\uZ\smas \GG/{C_{m_i}} _+\smas {S^\lambda}^{a_i-b_i}
			\to  H\uZ\smas  X\big(\T(L)+ \OO_i\big)_+
			\to H\uZ\smas X(\T(L))_+.
		\end{myeq}
Note that the summands of the right hand side of the form $H\uZ \wedge S^\lambda$ and $H\uZ \wedge {\GG/C_d}_+ \wedge S^{\lambda^{a_\mu-b_\mu}}$ support a trivial connecting map using Theorem \ref{thm:BG-C_m}, and the facts \\
a) $|(\lambda^r-\lambda^s-1)^K|> -1$ if $|K|\mid r$ but not $s$. \\
b) $|(\lambda^r-\lambda^s-1)^K|< -1$ if $|K|\mid s$ but not $r$. 

We now note that if the stabilizer $C_{m_v}$ of a vertex  $v$ satisfies $m_v < m$, then $m_v$ must divide $a_0$ or $b_0$ under the given hypothesis. In the case of type II trees, this is clear from (7)(b) of Definition \ref{adm-wt-tree}  as the maximum value of $m_v$ is reached among the vertices at level 1. For a type I tree, the analogous role is played by vertices $v$ with $m_v<m$ that are joined to a vertex $w$ of $\T_0$. The same condition now implies that one of $a_w$, $b_w$, $a_w-b_w$ is divisible by $m_v$. The hypothesis rules out the third case. Now we repeatedly apply condition (7)(a) of Definition \ref{adm-wt-tree} along the path from $w$ to the root vertex $v_0$ with the hypothesis ruling out the fact that $m_v$ divides $a_u - b_u$ for any vertex $u$ along the path. It follows that $m_v$ divides either $a_u$ or $b_u$ for every vertex along this path. Therefore, $m_v$ divides either $a_0$ or $b_0$. Now by Theorem \ref{thm:BG-C_m} using the fact that $m_i$ divides either $a_0$ or $b_0$, the connecting map on the summand $H\uZ \wedge S^{\lambda^{a_0}+\lambda^{b_0}}$ is $0$.     
%
		This completes the proof.
	\end{proof}
	
\begin{rmk}
Observe that Theorem \ref{thm:gen} has no hypothesis if the tree $\T$ is of type II. Henceforth, we prove further results for trees of type I. The hypothesis in Theorem \ref{thm:gen} is required crucially in the proof. For example, observe that if $ a_0-b_0\equiv 0\pmod m $, then the the cofibre sequence \eqref{cofseqcp2}  gives rise to the connecting map
	\[
	H\uZ \smas S^{\lambda^{a_0}+ \lambda^{b_0}}\xrightarrow{}  H\uZ \smas S^3, 
	\] 
	which  is determined by 
	\[
	\pi_{\lambda^{a_0}+\lambda^{b_0}-3}^\GG(H\uZ).
	\]
	This  group may be non-zero.
\end{rmk}

\vspace*{0.4cm}

 In the following theorem, we observe that if one of the rotation numbers at the root vertex is $ 0 $, then we obtain a decomposition result with no further hypothesis on the weights.	
	
	\begin{thm}\label{thm:one-zero}
		Let $ \T $ be an admissible weighted tree with $ \GG $-action of type I such that for the root vertex $ v_0 $ with  $ w(v_0)=(a_0,b_0;m) $, one of $ a_0 $ or $ b_0 $ is zero.  Then, $ H\uZ\smas X(\T)_+ $ admits the following decomposition
		\[
		H\uZ\smas X(\T)_+\simeq
			H\uZ\bigwedg H\uZ \smas S^{\lambda +2}\underset{[\mu]\in \T_d/\GG}{\bigvee} { H\uZ \smas \GG/C_d}_+\smas {S^{\lambda^{a_\mu-b_\mu}}}.
		\]
	\end{thm}

	\begin{proof}
		We proceed by induction on $L(\T)$, the maximum level reached by the tree as in Theorem \ref{thm:gen}. We may assume $ b_0$ is zero, so we have, $ \gcd(a_0,m)=1 $. At the initial case $L(\T)=0$, $X(\T)=\C P^2(a_0,0;m)  $, and for this manifold,  the cellular decomposition gives us	the following cofibre sequence of $ H\uZ$-modules
		\begin{myeq}\label{eq:base-one-zero}
			H\uZ\smas 	{S^\lambda}^{a_0} \to H\uZ\smas   \C P^2(a_0,0;m) \to H\uZ\smas  S^{\lambda^{a_0}+ 2}.
		\end{myeq}
		Note that $ H\uZ \smas S^{\lambda^{a_0}}\simeq H\uZ \smas S^{\lambda} $. The connecting  map   in \eqref{eq:base-one-zero},
		\[ 
		H\uZ \smas S^{\lambda+ 2}\xrightarrow{} H\uZ \smas S^{\lambda+ 1}
		\] 
		is trivial up to homotopy,	hence the cofibre sequence  splits and we obtain
		\[ 
		H\uZ \smas  \C P^2(a_0,0;m)_+\simeq H\uZ\bigwedg H\uZ \smas S^{\lambda +2}\bigwedg H\uZ \smas {S^\lambda}.
		\]
		For the inductive step, assume the statement is true for the tree up to level $ L $, $ \T(L) $ and we attach one orbit $ \OO_i $ of the level  $ L+1 $ vertices to $ \T(L) $. We prove the case when the  stabilizer for the vertices in $ \OO_i $ is  $ \GG $. An analogous reasoning applies to the other cases.\par
		Suppose  $ \OO_i= \{v_i\}$ and  $w(v_i)= (a_i,b_i;m) $.  As in the proof of Theorem \ref{thm:gen}, we get the following cofibre sequence
		\begin{equation*}
			{S^{\lambda}}^{a_i-b_i} \to X\big(\T(L)+ \OO_i\big)\to X\big(\T(L)\big).
		\end{equation*}
		Applying the induction hypothesis on $ X\big(\T(L)\big) $, we get a cofibre sequence of 
		$ H\uZ$-modules
		\begin{myeq}\label{eq:ind-one-zero}
			{H\uZ\smas S^\lambda}^{a_i-b_i} \to  H\uZ\smas  X\big(\T(L)+ \OO_i\big)\to  H\uZ \smas S^{\lambda +2}\bigwedg_{{[\mu]\in\T(L)_d/\GG}} { H\uZ \smas \GG/C_d}_+\smas {S^{\lambda}}^{a_\mu-b_\mu}.
		\end{myeq}
		We claim that the connecting map  is zero from each summand  of the right hand side of the equation   to $  H\uZ\smas  S^{\lambda^{a_i-b_i}+1}$. For the first summand note that the group 
		\[  
		\pi^\GG_{\lambda+1-\lambda^{a_i-b_i}}(H\uZ)= 0 
		\]
		as it	satisfies the condition given in  Theorem \ref{thm:BG-C_m}.
		To show the map
		\[  
		{ H\uZ \smas \GG/C_d}_+\smas {S^{\lambda}}^{a_\mu-b_\mu}\to H\uZ\smas  S^{\lambda^{a_i-b_i}+1}
		\]
		is trivial consider the  group 
		\[  
		\pi_{\alpha}^{C_d}(H\uZ), \quad \alpha= \lambda^{a_\mu-b_\mu}-\lambda^{a_i-b_i}-1 
		\]
		Then $\alpabs{}=-1  $, and  for all subgroups $ C_j $, $ \alpabs{C_j}\le 1 $. Equality holds if and only if $ j\mid (a_\mu-  b_\mu) $ and $ j\nmid (a_i-b_i) $. Then for any subgroup $ C_k \supset C_j $, $ k $ does not divide $ (a_i-b_i) $. So $ \alpabs{C_k}\ge-1 $. Thus $ \alpha $ satisfies the condition given in Theorem \ref{thm:BG-C_m}.
		Hence the cofibre sequence in \eqref{eq:ind-one-zero} splits and we obtain the required decomposition. This completes the proof.
	\end{proof}

	\section{Homology decompositions for $ \GG=C_{p^n} $}\label{sec:5}
	In this section we derive homology decompositions for tree manifolds in the case $ \GG=C_{p^n} $ without any restriction on weight. We start with an example pointing out the need for a judicious choice of cellular filtration and later, we discuss how a  reorientation may help to solve this. As a result, we observe some dimension shifting phenomena among  the summands  in the homology decomposition.    
	\begin{exam}\label{exam:non-triv-conn}
		Let $ p\mid a-b $. We know from the cellular filtration of projective spaces (Example \ref{cp2homdec}) that
		$ \C P^2(a,b;p)=P(\lambda^a\oplus  1_\C\oplus\lambda^b) $ gives us the cofibre sequence
		\begin{equation*}
			H\uZ\smas 	S^{\lambda^a}\to H\uZ\smas  \C P^2(a,b;m) \to H\uZ\smas  S^{\lambda^{a-b}+\lambda^{-b}}\simeq H\uZ\smas S^{\lambda+2}.
		\end{equation*}
		Since the connecting map is zero, we obtain 
		\begin{myeq}\label{eq:ex-cp-split}
			H\uZ\smas  \C P^2(a,b;m)\simeq  H\uZ\smas  S^{\lambda+2}\bigwedg H\uZ\smas 	{S^\lambda}.
		\end{myeq}
		One may also write $ \C P^2(a,b;p)=P(\lambda^a\oplus\lambda^b\oplus  1_\C) $ 
		which yields the cofibre sequence
		\begin{myeq}\label{eq:ex-cp-nontriv-cof}
		H\uZ\smas S^{\lambda^{a-b}}	\simeq H\uZ\smas 	{S^2}\to H\uZ\smas  \C P^2(a,b;m) \to H\uZ\smas  S^{2\lambda}.
		\end{myeq}
		We claim that the connecting map of \eqref{eq:ex-cp-nontriv-cof} is non-zero. Suppose on the contrary that  the connecting map is trivial. Then we have the splitting
		\begin{myeq}\label{eq:ex-cp-nontriv-split}
			H\uZ\smas  \C P^2(a,b;m)\simeq H\uZ\smas 	S^{2\lambda} \bigwedg H\uZ\smas  S^{2}.
		\end{myeq}
		The Mackey functor
		\[ \
		\upi^{\cp}_{\lambda}(H\uZ\smas  \C P^2(a,b;m)) \]
		is isomorphic \eqref{cp-comp} to $ \uZ $ from \eqref{eq:ex-cp-split}, while isomorphic to $ \uZ^*\oplus \bZp $ if  \eqref{eq:ex-cp-nontriv-split} were true, a contradiction. Hence, the connecting map of \eqref{eq:ex-cp-nontriv-cof} should be non-trivial.
	\end{exam}
	Recall from Proposition  \ref{tree-mfld-const} $ (4) $ that for a non-root vertex $ v\in V(\T) $, the connected sum is performed at the point $ [0,0,1]\in \C P^2_v $. At the root vertex $ v_0 $ with weight $ (a_0,b_0;m) $ we may change the weights to $ (-a_0,b_0-a_0;m) $ or $ (a_0-b_0,-b_0;m) $ to obtain a $ 
	\GG $-homeomorphic manifold. 
	 This fact will be used in the proof of the theorems below.
	
	The following example summarizes how a reorientation is performed and the resulting modifications of weights. We also demonstrate how this allows us to make a judicious choice of cellular filtration of $  X(\T) $. 
	\begin{exam}\label{ex:reori}
		Suppose we have an admissible weighted  $ \GG $-equivariant tree $ \T $ depicted as in the left of the figure  below with root vertex $ v_0 $, $ w(v_0)=(a_0,b_0;m) $, and all other vertices  have weight as mentioned therein.
		 We reorient $ \T $ to obtain a new tree $ \T' $ whose root vertex is $ u_0=v_3 $ with $ w(u_0) =(a_3-b_3,-b_3;m)$, and let for $ i=1, 2, 3 $, the vertices $ u_i\in \T' $ represent the vertices $ v_{3-i} $. Proposition  \ref{tree-mfld-const} $ (4) $ tells us  that the connected sum is performed at $ [0,0,1]\in \C P^2_{v_3} $ for which
		 \[
		T_{[0,0,1]}\C P^2_{v_3}=\lambda^{a_3}+\lambda^{b_3}.
		\]
		This means there exists a $ \GG $-fixed point $ p $ in $\C P^2_{v_2} $ so that 
		\[ 
		T_p\C P^2_{v_2}= \lambda^{a_3}+\lambda^{-b_3} \mbox{ or } \lambda^{-a_3}+\lambda^{b_3}.
		 \]
		 Then, if necessary, we apply a suitable $ \GG $-homeomorphism to map the point $ p $ to $ [0,0,1] $, which allow us to perform the connected sum at the point $ [0,0,1] $ of $ \C P^2_{u_1} $ with $ \C P^2_{u_0} $. This explains the weights in the new tree $ \T' $. Note that $ X(\T) $ is $\GG$ homeomorphic to $ X(\T') $. \par
		 \addtocounter{equation}{1}
		 \begin{figure}[htp]
		 	\begin{tikzpicture}
		 		[
		 		scale=.45,
		 		dot/.style = {circle, gray, minimum size=#1,
		 			inner sep=0pt, outer sep=0pt}, 
		 		dot/.default = 3pt  
		 		] 
		 		
		 		\node [dot, label=left:\up{$v_0:w(v_0)=(a_0,b_0;m)$}] (a0) at (3,6)  {};  
		 		\draw[black] (3,6) circle[radius=3pt];
		 		\node [dot, label=left:\up{$v_1:w(v_1)=(a_1,b_1;m)$}](a1) at (1,4)  {};
		 		\draw[black] (1,4) circle[radius=3pt];
		 		\node [dot, label=left:\up{$v_2:w(v_2)=(a_2,b_2;m)$\;\;\;}] (a2) at (3,2) {};  
		 		\draw[black] (3,2) circle[radius=3pt];
		 		\node [dot, label=left:\up{$v_3:w(v_3)=(a_3,b_3;m)$}] (a3) at (1,0) {}; 
		 		\draw[black] (1,0) circle[radius=3pt];

		 		\draw(a0) --  (a1); 
		 		\draw (a1) -- (a2);  
		 		\draw (a2) -- (a3);  
		 	\end{tikzpicture}
		 	\hspace*{2.5cm}
		 	\begin{tikzpicture}  
		 		[
		 		scale=.45,
		 		dot/.style = {circle, gray, minimum size=#1,
		 			inner sep=0pt, outer sep=0pt}, 
		 		dot/.default = 3pt  
		 		]  
		 		
		 		\node [dot, label=right:\up{$u_3:w(u_3)=\pm(a_1,-b_1;m)$}] (a0) at (3,6)  {};  
		 		\draw[black] (3,6) circle[radius=3pt];
		 		
		 		\node [dot, label=right:\up{\;$u_2:w(u_2)=\pm(a_2,-b_2;m)$}](a1) at (1,4)  {};
		 		\draw[black] (1,4) circle[radius=3pt];
		 		
		 		\node [dot, label=right:\up{$u_1:w(u_1)=\pm(a_3,-b_3;m)$}] (a2) at (3,2) {};  
		 		\draw[black] (3,2) circle[radius=3pt];
		 		
		 		\node [dot, label=right:\up{$u_0:w(u_0)=(a_3-b_3,-b_3;m)$}] (a3) at (1,0) {}; 
		 		\draw[black] (1,0) circle[radius=3pt];

		 		\draw (a0) -- (a1); 
		 		\draw (a1) -- (a2);  
		 		\draw (a2) -- (a3);  
		 	\end{tikzpicture}
		 	\label{fig}
		 	\caption{An example of reorientation: The left tree has  root vertex  $ v_0 $ and the right one has  root vertex $ u_0=v_3 $.} 
		 \end{figure}
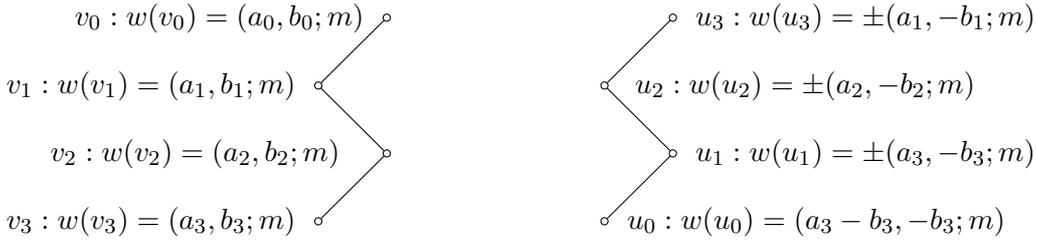\par
Suppose in the tree $ \T $,  $ m\nmid a_0, b_0, a_1-b_1, a_2-b_2 $ but $ m\mid a_3-b_3 $.	 For the $ \GG $-manifold $ X(\T) = \#_{i=0}^3\C P^2_{v_i}$ if we proceed as in the proof of Theorem \ref{thm:gen}, we obtain the cofibre sequence 
		\begin{equation*}
		\xymatrix@R=0.5cm@C=0.5cm{H\uZ\smas S^{\lambda^{a_3-b_3}} \ar[d]^{\simeq}\ar[r] 
			&H\uZ\smas X(\T) \ar[r] 
			& H\uZ\smas \#_{i=0}^2 \C P^2_{v_i}\ar[d]^{\simeq}\\ 
                 H\uZ\smas S^2 & & H\uZ\smas S^{2\lambda}\bigwedg_{i=0}^2 H\uZ\smas S^{\lambda}}
	    \end{equation*}
		Observe that the connecting map may be non-zero here. On the other hand, for $ \T' $, we obtain 
		\begin{myeq}\label{eq:reori-2}
		\xymatrix@R=0.5cm@C=0.5cm{H\uZ\smas S^{\lambda^{\pm(a_1+b_1)}} \ar[d]^{\simeq}\ar[r] 
			&H\uZ\smas X(\T') \ar[r] 
			& H\uZ\smas \#_{i=0}^2 \C P^2_{u_i}\ar[d]^{\simeq}\\ 
			H\uZ\smas S^{\lambda} & & H\uZ\smas S^{\lambda+2}\bigwedg_{i=0}^2 H\uZ\smas S^{\lambda}}
		\end{myeq}
The right vertical equivalence comes from Theorem \ref{thm:one-zero} and  the equivalence $ H\uZ\smas S^{\lambda^{a_i+b_i}}\simeq H\uZ\smas S^{\lambda} $  for $ i=1,2, 3 $. To see this, note from $ 7(a) $ of \ref{adm-wt-tree} that $$  \pm(a_1,b_1) \in \{(a_0,-b_0),(a_0-b_0,b_0), (a_0,b_0-a_0)\}. $$ Our condition $ m\nmid a_0, b_0, a_0-b_0 $ implies $ m\nmid a_1,b_1, a_1+b_1 $. Iterating this process the desired equivalence follows. Observe that, the connecting map in  \cofseq{} \eqref{eq:reori-2}   is trivial.\par
	\end{exam}

	We now proceed towards the decomposition result in the case $ \GG=C_p $. For the tree $ \T $, if $ p\nmid a_v-b_v $ for all vertices $ v \in \T_0$, the result is obtained from Theorem \ref{thm:gen} as 
	\[
	H\uZ\smas X(\T)_+ \simeq 
	H\uZ\bigwedg H\uZ \smas S^{\lambda^{a_0}+ \lambda^{b_0}}\bigwedg\limits_{\T_0} H\uZ \smas S^{\lambda}\bigwedg\limits_{ \T_e/\GG} { H\uZ \smas \GG/e}_+\smas {S^2}. 
	\]
	where $ \T_e=\{v\in \T\mid Stab(v)=e\} $. In the complementary situation $ p $ must divide $ a_v-b_v $ for some $ v\in \T_0 $. If further $ p\mid a_0 $ or $ b_0 $, we are in the situation dealt in Theorem \ref{thm:one-zero}, so that we have 
	\[
		H\uZ\smas X(\T)_+ \simeq 
	H\uZ\bigwedg H\uZ \smas S^{{\lambda}+ 2}\bigwedg\limits_{\phi(\T)} H\uZ \smas S^{\lambda}\bigwedg\limits_{\psi(\T)}H\uZ \smas {S^2}\bigwedg\limits_{ \T_e/\GG} { H\uZ \smas \GG/e}_+\smas {S^2}. 
	\] 
	where  $ \phi(\T)=\#\{v\in \T_0\mid p\nmid a_v-b_v\} $, and $ \psi(\T)=\#\{v\in \T_0\mid p\up{divides} a_v-b_v\} $. For the remaining case, we prove
	\begin{theorem}\label{thm:cp-case}
		Let  $ \GG=C_p $ and $\T $ be an admissible weighted  $ \GG $-equivariant tree of type I such that $ p\nmid a_0, b_0 $ but $ p\mid a_v-b_v $ for some $ v\in \T_0 $.  	Then,  
		\begin{equation*}
				H\uZ\smas X(\T)_+ \simeq 
				H\uZ\bigwedg H\uZ \smas S^{{\lambda}+ 2}\bigwedg\limits_{\phi(\T)+1} H\uZ \smas S^{\lambda}\bigwedg\limits_{\psi(\T)-1}H\uZ \smas {S^2}\bigwedg\limits_{ \T_e/\GG} { H\uZ \smas \GG/e}_+\smas {S^2}. 
		\end{equation*}
	\end{theorem}
	\begin{proof}
By given hypothesis,  there exist a vertex $ v_\ell\in \T_0 $ with $ w(v_\ell)=(a_\ell,b_\ell;p) $ such that $  a_\ell-b_\ell \equiv 0 \pmod p $.  We  further assume that $ v_\ell $ is closest to the root vertex $ v_0 $ in terms of number of edges from $ v_0 $ to $ v_\ell $. We  prove  the statement for  $\T_0\subseteq \T $. The result follows from this because  when  we attach a free orbit  to any level, the resulting  connecting map becomes a map of underlying non-equivariant spectra, which is trivial.\par
 Let $ \Gamma $ denote  the path from $ v_0 $ to $ v_\ell $ passing through vertices $ v_0, v_1,\dots ,v _\ell $, and let  for  $ v_i\in \Gamma $,  $w(v_i)= (a_i,b_i;p) $. We  reorient $ \Gamma $, as in Example \ref{ex:reori}, so that now $ v_\ell $ becomes the root vertex $ u_0 $; $ v_i\in \Gamma $ becomes the vertex $ u_{\ell-i} $.  Observe that  the weight at  $ u_{i}=v_{\ell-i} $ becomes $ \pm(a_{\ell-i+1},-b_{\ell-i+1};p) $ and the weight at $ u_0 $ is  $ (a_\ell-b_{\ell}, -b_\ell;p) $.  
		Since $ a_\ell-b_\ell \equiv 0 \pmod p$,  the choice of cellular filtration of $ \C P^2(a_\ell,b_\ell;p) $, as in Example \ref{exam:non-triv-conn}, leads to 
		\[  H\uZ \smas  \C P^2(a_\ell,b_\ell;p)\simeq H\uZ \smas S^{ \lambda+2}\bigwedg H\uZ \smas S^{\lambda}.\]
	Theorem \ref{thm:one-zero} implies 
\begin{align*}
H\uZ\smas X(\Gamma)	 & \simeq H\uZ\smas  S^{\lambda+2}\bigwedg H\uZ\smas 	{S^\lambda}\bigwedg_{i=1}^{\ell} H\uZ\smas {S^\lambda}^{a_{\ell-i+1}+b_{\ell-i+1}}\\
	&\simeq   H\uZ\smas  S^{\lambda+2}\bigwedg H\uZ\smas 	{S^\lambda}\bigwedg_\ell H\uZ\smas 	{S^\lambda}
\end{align*}
	To deduce the second equivalence we show that for $ i=1 $ to $ \ell $, $ p \nmid a_i+b_i $. 
	The condition  $ p\nmid  a_0 $, $ b_0,  a_0-b_0 $ in turn implies $ p\nmid  a_1,b_1, a_1+b_1 $ as  by $ 7(a) $ of  Definition \ref{adm-wt-tree}  $ \pm(a_1,b_1) \in \{(a_0,-b_0),(a_0-b_0,b_0), (a_0,b_0-a_0)\}$. The fact that for $ 1\le i\le \ell $, $ p\nmid a_{i-1}-b_{i-1} $, let us continue this process up to ${\ell-1}$. \par
Next we attach vertices of $ \T\setminus \Gamma  $ to $ \Gamma $	proceeding  by induction  as in Theorem $ \ref{thm:one-zero} $, and observe that in each step the connecting map for the \cofseq{} is null. This completes the proof.
	\end{proof}
	The rest of the section is devoted to proving homology decompositions in the case  $ \GG=C_{p^n} $. We start with the following observation 
	\begin{lemma}\label{lem:cpn-lem}
	Let $ \GG=C_{p^n} $	 and $\T $ be an admissible weighted   $ \GG $-equivariant tree of type I. Let $ \tau $ be the maximum power of $ p $ that divides $ a_v-b_v $ among the vertices of $ \T_0 $, and  $ p^\tau$ does not divide  $a_0$ and $ b_0 $. Then for $ v\in \T\setminus \T_0 $, $Stab(v)\le C_{p^\tau} $.
	\end{lemma}
\begin{proof}
	Choose  a vertex $v\in \T\setminus \T_0  $  such that $  Stab(v)$ is maximum among vertices of $ \T\setminus\T_0 $, and $ v $ is closest to the root vertex. 
	This implies if $ v $ is in level $ L $, then the vertex $ u $  in  level $ L-1 $ connected to $ v $ belongs to $ \T_0 $.   So let $ w(u)=(a_u,b_u,p^n) $. We claim that $m_v=Stab(v)\le C_{p^\tau} $.
	 On the contrary suppose
	$m_v> C_{p^\tau} $. From $ 
	7(b) $ of Definition \ref{adm-wt-tree}, we get 
	$ m_v$ divides one of $ a_u,b_u, a_u-b_u $. Since $ m_v$ can not divide $  a_u-b_u  $, $ m_v $ divides $ a_u $ or $ b_u $. Further, if $ u $ is connected to a level $ L-2$ vertex $ u' $  with $ w(u') =(a_{u'},b_{u'};p^n)$, then 
	\[ 
	\pm(a_u,b_u)\in\{(a_{u'},-b_{u'}), (a_{u'},b_{u'}-a_{u'}), (a_{u'}-b_{u'},b_{u'})\}.
	 \]
	 This means $ m_v$ divides $ a_{u'} $ or $ b_{u'} $. Continuing this process we end up with $ p^\tau $ dividing $ a_0 $ or $ b_0 $. Hence a contradiction.
\end{proof}
Proceeding as in the $ C_p $-case, we have that if $ p\nmid a_v-b_v $ for all $ v \in \T_0$, by Theorem \ref{thm:gen}
\[
	H\uZ\smas X(\T)_+\simeq H\uZ\bigwedg H\uZ \smas S^{{\lambda^{a_0}}+ \lambda^{b_0}}\bigwedg\limits_{\T_0} H\uZ \smas S^{\lambda}\bigwedg\limits_{{\T_d/\GG},~d\neq p^n}  H\uZ \smas {\GG/C_d}_+\smas S^{\lambda}.
\]
 In the complementary case $p  $  divides $ a_v-b_v $ for some $ v\in \T $. Let $ \tau>0 $ be the maximum power of $ p $ that divides $ a_v-b_v $ among the vertices of $ \T_0 $. If  further $ p^\tau\mid a_0$ or $ b_0 $, then proceeding exactly as in Theorem \ref{thm:one-zero} we obtain 
\begin{myeq}\label{eq:cpn-ref}
H\uZ\smas X(\T)_+\simeq H\uZ\bigwedg H\uZ \smas S^{{\lambda^{a_0}}+ \lambda^{b_0}}\bigwedg\limits_{Z_\T(i)} H\uZ \smas S^{\lambda^{i}}\bigwedg\limits_{{[\mu]\in \T_d/\GG,~ d\neq p^n}}  H\uZ \smas {\GG/C_d}_+\smas S^{\lambda^{a_\mu-b_\mu}}
\end{myeq}
where
 $
Z_\T(i):= \#\{v\in \T_0, w(v)=(a_v,b_v;p^n) \mid \gcd(a_v-b_v,p^n)=p^i \mbox{ where }  0\le i\le n \}. 
$\par
We also define
	\[
	W_{\T}(i)=\begin{cases}
		Z_\T(i) +1  & \mbox{ if } i=0 \\ 
		Z_\T(i)-1 & \mbox{ if } i=\tau\\
		Z_\T(i) & \mbox{ otherwise.}  \mbox{}\\	  
	\end{cases}
	\]
	Observe that the conditions on weights in Definition \ref{adm-wt-tree} does not change if we replace the weight $ (a_0,b_0;p^n) $ at the root vertex by one of $ (a_0-b_0,-b_0;p^n) $ or $ (b_0-a_0,-a_0;p^n) $. This allows us to further assume $ p\nmid a_0,b_0 $ in the theorem below.
	\begin{thm}\label{cpnthm}
		Let  $ \GG=C_{p^n}$ and $\T $ be an admissible  weighted  $ \GG $-equivariant  tree of type I.  Suppose $ \tau>0 $ is the maximum power of $ p $ that divides $ a_v-b_v $ among the vertices of $ \T_0 $ and  $ p\nmid a_0, b_0 $. Then  
		\begin{equation*}
			H\uZ\smas X(\T)_+\simeq	H\uZ\bigwedg H\uZ \smas S^{\lambda+\lambda^{p^{\tau}}}\bigwedg\limits_{W_{\T}(i)} H\uZ \smas S^{\lambda^{i}}\bigwedg\limits_{{[\mu]\in \T_d/\GG,~d\neq p^n}} H\uZ \smas {\GG/C_d}_+\smas S^{\lambda^{a_\mu-b_\mu}}. 
		\end{equation*}
	\end{thm}

	\begin{proof}
	 Let $ v_\ell \in \T_0$ with $ w(v_\ell)=(a_\ell,b_\ell;p^n) $ be a vertex for which  $ p^\tau\mid a_\ell-b_\ell $, and $ v_\ell $ is closest to the root in terms of number of edges. Let $ \Gamma $ denote the path from $ v_0 $ to $ v_\ell $ passing through the vertices $ v_0, v_1, \dots, v_\ell $ with $ w(v_i) =(a_i,b_i;p^n)$. We first compute $ H\uZ\smas X(\Gamma) $, then $H\uZ\smas X(\T_0)$  and finally $ H\uZ\smas X(\T) $.\par 
	We reorient $ \Gamma $ so that now $v_\ell  $ becomes the root vertex $ u_0 $, and $ v_{\ell-i} $ becomes the vertex $ u_{i} $. Observe that  the  weight at $ u_{i}=v_{\ell-i} $ becomes $ \pm(a_{\ell-i+1},-b_{\ell-i+1};p^n) $ and the weight at $ u_0 $ is $ (a_\ell-b_\ell ,-b_\ell;p^n) $. Since  $\gcd(a_\ell-b_\ell,p^n)=p^\tau$,
		 \begin{myeq}\label{eq:cpn-reori-root}
		  {H\uZ\smas  S^\lambda}^{a_\ell-b_\ell} \simeq { H\uZ\smas  S^{\lambda^{p^\tau}}}.
		 \end{myeq}
		 We use  the following cofibre sequence for $ \C P^2(a_\ell,b_\ell;p^n) $
		 \[ 
		 S^{\lambda^{a_\ell}} \to  \C P^2(a_\ell,b_\ell;p^n)\to S^{\lambda^{a_\ell-b_\ell}+\lambda^{-b_\ell}}.
		 \]
		This together with the identification \eqref{eq:cpn-reori-root}  leads to the following decomposition
		\begin{equation*}
			H\uZ\smas   \C P^2(a_\ell,b_\ell;p^n)_+\simeq H\uZ\bigwedg H\uZ\smas  S^{\lambda+\lambda^{p^\tau}}\bigwedg H\uZ\smas 	{S^\lambda}.
		\end{equation*}
Proceeding as in the proof of Theorem \ref{thm:cp-case} we see that for $ i=0 $ to $ \ell-1 $, $ p^\tau\nmid a_{i+1}+b_{i+1} $. 
Now \eqref{eq:cpn-ref}  implies 
\begin{align*}
	H\uZ\smas X(\Gamma)&\simeq  H\uZ\smas  S^{\lambda+\lambda^{p^\tau}}\bigwedg H\uZ\smas 	{S^\lambda} \bigwedg_{i=1}^{\ell} H\uZ\smas {S^\lambda}^{a_{\ell-i+1}+b_{\ell-i+1}}\\
	&\simeq H\uZ\smas  S^{\lambda+\lambda^{p^\tau}} \bigwedg\limits_{W_{\Gamma}(i)} H\uZ \smas S^{\lambda^{i}},
\end{align*}
The last equivalence is obtained from Proposition \ref{smash-rep} and the following claim\par 
\textbf{Claim:}\label{lem:card-eq}
 	Given $ \tau  $  and $ \Gamma $ as above, we have for $ s<\tau $
 	\[
 	\#\{1\le i\le \ell\mid \nu_p(a_i+b_i)=s\}=	
 	\#\{0\le j\le \ell-1\mid \nu_p(a_j-b_j)=s\},
 	\]
 	where $ \nu_p(r)=\mbox{max} \{ k\mid p^k \up{divides}  r\} $ is the $ p $-adic valuation.

{\it Proof of the claim.}
	Suppose $ p^s $ divides $ a_i+b_i $ for some  $ 1\le i\le \ell $ and $ v_{i-1} $ is not the root vertex. Applying $ 7(a) $  of Definition \ref{adm-wt-tree}, we see  
	  $ p^s $ divides $ a_{i-1}$ or $b_{i-1} $ (this implies $p\nmid  a_{i-1}-b_{i-1}$). Going one step further we see $ p^s $ divides one of $ a_{i-2}-b_{i-2},  a_{i-2} \up{or} b_{i-2} $, and the other two are relatively prime to $ p $. Since $ p\nmid a_0,b_0 $, continuing this process we end up with $ p^s$ dividing  $a_j-b_j $ for  some $ 0\le j< i-1 $ and $ p\nmid a_q-b_q $ or $ a_q+b_q $ for $ j<q\le i-1 $. If $ p^s\mid a_1+b_1 $, we must have $ p^s\mid a_0-b_0 $ by $ (5) $ of  \ref{adm-wt-tree} and the fact that $ p\nmid a_0,b_0 $. Therefore, the left-hand side is less than or equal to right-hand side.\par
	In the reverse direction, suppose  $ p^s \mid a_j-b_j $ for some $ 0\le j\le \ell-1 $ and $ v_j $ is not the root vertex. Then by $ 7(a) $ of Definition \ref{adm-wt-tree}, we see  $p^s $ divides one of $a_{j+1}, b_{j+1}   $, and hence, $p$ divides neither of $a_{j+1}+b_{j+1}$ or $a_{j+1}-b_{j+1}$. Continuing further we see that $ p^s $ divides one of $ a_{j+2}+b_{j+2},  a_{j+2} \up{or} b_{j+2} $.  Since $ p\mid a_\ell-b_\ell  $, $ p\nmid a_\ell,b_\ell $. Thus iterating this process  we see  $ p^s$  divides $ a_{i}+b_{i} $ for some $ j+1< i<\ell $ and $ p\nmid a_q+b_q $ or $ a_q-b_q $ for $ j+1\le q< i $. If $ p^s\mid a_0-b_0 $, then by $ (5) $ of Definition \ref{adm-wt-tree}, $ p^s $ divides one $ a_1+b_1,a_1 $ or $ b_1 $. If $ p^s\nmid a_1+b_1 $, then continuing one step further we see $ p^s $ divides one of $ a_2+b_2 , a_2$ or $ b_2 $. Iterating this way we obtain the required. This completes the proof of the claim.\par 
 Next we  attach vertices of $ \T_0\setminus \Gamma $ to $ \Gamma $  proceeding by induction as in the proof of Theorem \ref{thm:one-zero}, except the fact  that now levels are defined according to the distance from $ \Gamma $ instead of the root vertex.  For the inductive step suppose the statement holds for the tree up to level $ L $, $ \T_0(L) $ and we adjoin an orbit $ \OO_x $ containing a  level $ L+1 $-vertex $ v_x $ to $ \T_0(L) $. Suppose  $ w(v_x)=(a_x,b_x;p^n) $.  By Proposition \ref{smash-rep}, we may write  $ H\uZ\smas S^{\lambda^{a_x-b_x}}\simeq H\uZ\smas S^{\lambda^{p^t}} $ for some $ 0\le t<  \tau $. Then we obtain the following cofibre sequence of $ H\uZ $-modules after applying the induction hypothesis
\[
 H\uZ\smas S^{\lambda^{p^t}} \to  H\uZ\smas  X(\T(L)+ \OO_x)\to  H\uZ \smas S^{\lambda+\lambda^{p^{\tau}}}\bigwedg\limits_{W_{\T(L)}(i)} H\uZ \smas S^{\lambda^{i}}.
 \]
  The fact that $ p^\tau $ is the highest power ensures the connecting maps are zero by Theorem \ref{thm:BG-C_m}. Hence the above cofibre sequence splits and the required homology decomposition is obtained.\par
 Finally, to complete the proof, we adjoin vertices of $ \T\setminus \T_0 $ to $ \T_0 $, i.e., vertices on which $ \GG $ acts non-trivially. Again we proceed by induction on levels where levels are defined according to the distance from $ \T_0 $. Assume the statement holds for the tree up to level $ L' $, $ \T(L') $ and we attach an orbit $ \OO_{y} $ of the level $ L'+1 $-vertices  to $ \T(L') $.  Suppose for $ v_{y}\in \OO_{y} $, $ w(v_{y})=(a_y,b_y;C_{m_y}) $, and  $ H\uZ\smas S^{\lambda^{a_y-b_y}}\simeq H\uZ\smas S^{\lambda^{p^{t'}}} $ for some $ 0\le t'\le n $.  Proceeding along the lines of Theorem \ref{thm:gen}, we obtain the following cofibre sequence of $ H\uZ $-modules
 	\begin{myeq}\label{eq:cpn-ind-nontriv}
 		H\uZ\smas \GG/{C_{m_y}} _+\smas {S^\lambda}^{p^{t'}}
 		\to  H\uZ\smas  X\big(\T(L')+ \OO_y\big)_+
 		\to H\uZ\smas X(\T(L'))_+
 	\end{myeq}
 where by the induction hypothesis 
 \[H\uZ\smas  X(\T(L'))_+\simeq H\uZ\smas H\uZ \smas S^{\lambda+ \lambda^{p^\tau}}\bigwedg\limits_{W_{\T}(i)} H\uZ \smas S^{\lambda^{i}}\bigwedg_{[\mu]\in\T(L)_d/\GG} { H\uZ \smas \GG/C_d}_+\smas {S^{\lambda^{a_\mu-b_\mu}}}. \]
 The connecting map 
 \[ H\uZ \smas S^{\lambda+ \lambda^{p^\tau}}\to H\uZ\smas \GG/{C_{m_y}} _+\smas S^{\lambda^{p^{t'}}+1 }
  \]
   is classified up to homotopy by $ \pi_0^{C_{m_y}} (H\uZ\smas S^{\lambda^{p^{t'}}+1-\lambda-\lambda^{p^\tau}})$. Since  Lemma \ref{lem:cpn-lem} asserts that  $ C_{m_y}\le C_{p^\tau} $, the above homotopy group reduces to $ \pi_0^{C_{m_y}} (H\uZ\smas S^{\lambda^{p^{t'}}-\lambda-1})$, which is trivial by Theorem \ref{thm:BG-C_m}. Analogously all other connecting maps can be seen to be zero. Hence the cofibre sequence \eqref{eq:cpn-ind-nontriv} splits and we obtain the required decomposition.
\end{proof}

\end{document}